\documentclass[12pt,amstags,fleqn]{article}

\usepackage{amssymb}

\usepackage{amsmath}
\usepackage{amsthm}

\theoremstyle{plain}
\newtheorem{theorem}{Theorem}[section]
\newtheorem{lemma}[theorem]{Lemma}
\newtheorem{proposition}[theorem]{Proposition}
\newtheorem{corollary}[theorem]{Corollary}

\theoremstyle{definition}
\newtheorem{definition}[theorem]{Definition}

\theoremstyle{definition}
\newtheorem{example}[theorem]{Example}
\newtheorem{remark}[theorem]{Remark}

\newcommand{\mov}{\textnormal{Mov}}

\newcommand{\darts}{\Omega}

\newcommand{\opa}{\diamond} 
\newcommand{\opb}{\otimes} 
\newcommand{\opbvar}{\times}

\newcommand{\autotopismgroup}[1]{\textnormal{Aut}(#1)}
\newcommand{\automorphismgroup}[1]{\textnormal{Aut}(#1)}

\usepackage{ifpdf}

\ifpdf
        \usepackage[pdftex]{graphicx}
\else
    \usepackage{epsfig}
        \usepackage{graphicx}
\fi

\title{Transitive latin bitrades}

\author{Carlo
H\"{a}m\"{a}l\"{a}inen\footnote{\texttt{carlo.hamalainen@gmail.com}} \\
Department of Mathematics \\
Charles University \\
Sokolovsk\'a 83, 186 75 Praha 8 \\
Czech Republic \\ 
~\\
Nicholas J. Cavenagh\\
School of Mathematical Sciences\\
Monash University\\
Vic 3800 Australia }

\begin{document}

\maketitle

\begin{abstract} 
In this note we give two results. First, if a latin 
bitrade $(T^{\opa},\,
T^{\opb})$ is primary, thin, separated, and 
$\autotopismgroup{T^{\opa}}$ acts regularly on $T^{\opa}$,
then
$(T^{\opa},\, T^{\opb})$ may be derived from a group-based construction. Second,
if a latin bitrade $(T^{\opa},\, T^{\opb})$ has genus~$0$ then 
the disjoint mate $T^{\opb}$ is unique
and the autotopism group of $T^{\opa}$ is equal to the autotopism group of
$T^{\opb}$.
\end{abstract}

\section{Introduction}

Here we collect some definitions and known results about latin bitrades.
We refer the reader to \cite{groupsarxiv,hamalainen2007,3hom} for further
background material (especially for group-based bitrades constructing a
hypermap from a bitrade.  

\begin{definition}\label{defnLatinSquare}
Let $A_1$, $A_2$, and $A_3$ be finite sets of size $n > 0$.
A {\em latin square\/} 
$L = L^{\opa}$ of {\em order $n$\/} is an $n \times n$ array with
rows indexed by $A_1$, columns indexed by $A_2$,
and entries from $A_3$.
Further, each $e \in A_3$ appears exactly once in each row
and exactly once in
each column.
A {\em partial latin square\/} 
of order $n$
is an $n \times n$ array where each $e \in A_3$ occurs at most once
in each row and at most once in each column. 
\end{definition}
We may view $L^{\opa}$ as a set and write $(x,\, y,\, z) \in L^{\opa}$
if and only if symbol $z$ appears in the cell at row $x$, column $y$.
As a binary operation we write
$x \opa y = z$ if and only if $(x,\, y,\, z) \in L^{\opa}$.

\begin{definition}
Let $A_1$, $A_2$, and $A_3$ be the row, column, and symbol labels,
respectively.  Let $T^{\opa}$, $T^{\opb}\subset A_1\times A_2\times A_3$ be two
partial latin squares. 
Then $(T^{\opa},\, T^{\opb})$ is called a {\em latin bitrade}
if the following  three
conditions are all satisfied:
\begin{itemize}
\item[(R1)] $T^{\opa} \cap T^{\opb} = \emptyset$.

\item[(R2)] For all $(a_1, a_2, a_3) \in T^{\opa}$ and all $r,s \in \{1, 2, 3\}$,
$r \neq s$, there exists a unique $(b_1, b_2, b_3) \in T^{\opb}$
such that $a_r=b_r$ and $a_s=b_s$.

\item[(R3)] For all $(a_1, a_2, a_3) \in T^{\opb}$ and all $r,s \in \{1, 2, 3\}$,
$r \neq s$, there exists a unique $(b_1, b_2, b_3) \in T^{\opa}$
such that $a_r=b_r$ and $a_s=b_s$.

\end{itemize}
\end{definition}

Two (partial) latin squares $P^{\opa}$ and $Q^{\opa}$ are {\em isotopic}
if there is some permutation 
$(\alpha,\beta,\gamma) \in S_n \times S_n \times S_n$ such that
\[
Q^{\opa} = \{ (i \alpha, j \beta, k \gamma) \mid (i, j, k) \in P^{\opa}
\}.
\]
We say that $(T^{\opa},\, T^{\opb})$ is isotopic to $(U^{\opa},\,
U^{\opb})$ if and only if $T^{\opa}$ is isotopic to $U^{\opa}$ and
$T^{\opb}$ is isotopic to $U^{\opb}$.

Trades have been studied for other types of combinatorial structures
(see~\cite{MR2041871} for a survey about trades in combinatorial designs
and~\cite{cavenaghSurvey} for a recent survey on latin trades). This
paper only refers to trades of latin squares so we abbreviate
latin bitrade to just {\em bitrade}. If $(T^{\opa},\, T^{\opb})$ is a
bitrade then we refer to $T^{\opa}$ as the {\em trade} and 
$T^{\opb}$ as the {\em disjoint mate}.

\begin{definition}\label{defya}
Define the map $\beta_r : T^{\opb} \rightarrow
T^{\opa}$ where $(a_1,a_2,a_3)\beta_r = (b_1,b_2,b_3)$ implies that
$a_r \neq b_r$
and $a_i = b_i$ for $i \neq r$.
(Note that by conditions \textnormal{ (R2)} and \textnormal{ (R3)} the map $\beta_r$ and its
inverse are well defined.) 
In particular, let $\tau_1,\tau_2,\tau_3: T^{\opa} \rightarrow T^{\opa}$, where
$\tau_1=\beta_2^{-1}\beta_3$,
$\tau_2=\beta_3^{-1}\beta_1$ and
$\tau_3=\beta_1^{-1}\beta_2$.
For each $i\in \{1,2,3\}$, let ${\mathcal A}_i$ be the set of cycles in $\tau_i$. 
\end{definition}
A latin bitrade is {\em separated} if
$\left| \mathcal{A}_i \right| = 1$ for $i = 1$, $2$, $3$.
In other words, each row, column, and symbol corresponds to a single
cycle of $\tau_1$, $\tau_2$, and $\tau_3$ respectively.

\begin{lemma}
\label{perm}
The permutations $\tau_1$, $\tau_2$ and $\tau_3$ satisfy the following properties:
\begin{itemize}
\item[\textnormal{(Q1)}] If $\rho \in {\mathcal A}_r$, $\mu \in {\mathcal A}_s$, $1 \leq r < s \leq 3$,
then $\left| \mov(\rho) \cap \mov(\mu) \right| \leq 1$.

\item[\textnormal{(Q2)}] For each $i \in \{1, 2, 3\}$, $\tau_i$ has no fixed points.
\item[\textnormal{(Q3)}] $\tau_1\tau_2\tau_3=1$.
\end{itemize}
\end{lemma}

\begin{definition}
A latin bitrade $(T^{\opa},\, T^{\opb})$ is said to be {\em primary}  
if whenever $(U^{\opa},\, U^{\opb})$ is a latin bitrade such that
$U^{\opa}\subseteq
T^{\opa}$ 
and
$U^{\opb}\subseteq   
T^{\opb}$, then 
$(T^{\opa},\, T^{\opb})=  
(U^{\opa},\, U^{\opb})$.  
\end{definition}  

\begin{definition}\label{defn:thin}
A latin bitrade $(T^{\circ}, T^{\star})$ is said to be {\em thin}  
if whenever $i\circ j = i'\circ j'$ (for $i\neq i'$, $j\neq j'$), then
$i\star j'$ is either undefined, or 
$i\star j'= i\circ j$.
\end{definition}

So if $(T^{\opa},\, T^{\opb})$ is thin and
$(T^{\opa},\, T^{\opbvar})$ is another bitrade, then
$T^{\opb} = T^{\opbvar}$.

Let $G = \langle \tau_1, \tau_2, \tau_3 \rangle$. Bitrades for which
$\left| G \right| = \left| T^{\opa} \right|$ have been studied in
\cite{groupsarxiv}. Informally, three group elements 
$a$, $b$, $c \in G$ may be identified as $\tau_1$, $\tau_2$, and
$\tau_3$. These elements act on $G$ by multiplication on the right.
Using just $\tau_1$, $\tau_2$, and $\tau_3$ we may recover a bitrade
$(U^{\opa},\, U^{\opb})$ which is isotopic to the original bitrade
$(T^{\opa},\, T^{\opb})$.

\begin{definition}\label{tdef}
Let $G$ be a finite group. 
Let $a$, 
$b$, $c$ be non-identity elements of $G$ and let
$A=\langle a \rangle$,
$B=\langle b \rangle$ and
$C=\langle c \rangle$. Define three conditions:
\begin{itemize}
\item[\textnormal{ (G1)}] $abc=1$;

\item[\textnormal{ (G2)}] $|A\cap B| = |A \cap C| = |B\cap C|=1$;

\item[\textnormal{ (G3)}] $\langle a,b,c \rangle = G$.

\end{itemize}
If \textnormal{(G1)}
and \textnormal{(G2)} are satisfied then define two 
(partial) arrays $T^{\opa}$
and $T^{\opb}$ by
\[
T^{\circ} = \{ (gA,gB,gC)\mid g\in G\}, \quad
T^{\star} = \{ (gA,gB,ga^{-1}C)\mid g\in G\}. 
\]
\end{definition}

\begin{theorem}[Dr\'apal \cite{Dr9}]\label{thmGroupBitrade}
The pair of (partial) arrays
$(T^{\circ},\, T^{\star})$ given in Definition~\ref{tdef}
form a latin bitrade with size $|G|$,
$|G:A|$ rows (each with $|A|$ entries), 
$|G:B|$ columns (each with $|B|$ entries)  and
$|G:C|$ entries (each occurring $|C|$ times).
If \textnormal{(G3)} is also satisfied then the
bitrade is primary.
\end{theorem} 

A group $G$ acting on a (finite) set $X$ is {\em regular} if for any
$x$, $y \in X$ there is a unique $g \in G$ such that
$xg = y$. A corollary is that $G$ acts transitively on $X$.

Euler's {\em genus formula} is:
\begin{equation}\label{eqnEuler}
2-2g = v - e + f.
\end{equation}
If the bitrade $(T^{\opa},\, T^{\opb})$ has $\tau_i$ representation
$[\tau_1,\, \tau_2,\, \tau_3]$ then the genus formula for the hypermap
embedding becomes
\[
2 - 2g = z(\tau_1) + z(\tau_2) + z(\tau_3) - \left| T^{\opa} \right|
\]
where $z(\tau_i)$ is the number of cycles in the permutation $\tau_i$.

Notation and definitions for autotopisms and automorphisms are as
follows:

\begin{itemize}

\item Autotopism group of a partial latin square of order $n$:
\[
\autotopismgroup{T^{\opa}} = \langle g \in S_n \times S_n \times S_n
\mid T^{\opa} g = T^{\opa} \rangle.
\]

\item Autotopism group of a bitrade of order $n$:
\[
\autotopismgroup{T^{\opa}, T^{\opb}} = \langle g \in S_n \times S_n \times S_n
\mid T^{\opa} g = T^{\opa} \textnormal{ and } T^{\opb} g = T^{\opb} \rangle.
\]
Equivalently, $\autotopismgroup{T^{\opa}, T^{\opb}} =
\autotopismgroup{T^{\opa}} \cap \autotopismgroup{T^{\opb}}$.

\item
A latin bitrade is transitive if
\begin{enumerate}
\item for each pair $(i,j,k),\, (i',\, j',\, k') \in T^{\opa}$, there is
an autotopism $(\alpha_1,\, \alpha_2,\, \alpha_3) \in 
\autotopismgroup{T^{\opa},\, T^{\opb}}$ such that 
$(i\alpha_1,\, j\alpha_2,\, k\alpha_3)=(i',\, j',\, k')$; and

\item for each pair $(i,j,k),\, (i',\, j',\, k') \in T^{\opb}$, there is
an autotopism $(\alpha_1,\, \alpha_2,\, \alpha_3) \in 
\autotopismgroup{T^{\opa},\, T^{\opb}}$ such that 
$(i\alpha_1,\, j\alpha_2,\, k\alpha_3)=(i',\, j',\, k')$.
\end{enumerate}

\item Automorphism group of a $\tau_i$ representation:
\[
\automorphismgroup{[\tau_i]} =
\automorphismgroup{[\tau_1,\, \tau_2,\, \tau_3]}
= \langle g \in S_m \mid  g \tau_i = \tau_i g \textnormal{ for $i=1$,
$2$, $3$} \rangle
\]
where $m = \left| T^{\opa} \right|$.  
So
$\automorphismgroup{[\tau_1,\, \tau_2,\, \tau_3]}$ is 
$C_{S_m}(G)$, the centraliser of $G$ in $S_m$.

\end{itemize}

\section{Transitive bitrades}

The following theorem provides necessary conditions for a bitrade to be
derived from a group:

\begin{theorem}\label{theoremSharplyTransitiveTrade}
Let $(T^{\opa},\, T^{\opb})$ be a 
primary, thin, separated bitrade such
that $\autotopismgroup{T^{\opa}}$ is regular.
Then $(T^{\opa},\, T^{\opb})$ is isotopic to the trade constructed 
as in Definition~\ref{tdef} using the group $\autotopismgroup{T^{\opa}}$.
\end{theorem}

\begin{proof}
Let $(T^{\opa},\, T^{\opb})$ be a nontrivial bitrade with the specified
properties.  
Fix an entry $(i,j,k) \in T^{\opa}$. 
Since $(T^{\opa},\, T^{\opb})$ is 
a bitrade,
there exists a $k'\neq k$ such that 
$(i,j',k')$ and $(i',j',k')$ are entries in 
$T^{\opa}$, for some $i'\neq i$ and $j'\neq j$.
The situation is as follows:
\[
\begin{array}{c|ccc}
\opa/\opb & j    & ~ & j' \\
\hline  i & k & ~ & k'/k \\
        ~ & ~    & ~ & ~  \\
       i' & ~   & ~ & k
\end{array}
\]
The bitrade is thin so the value $k'$ is the only possible choice for
$(i,j',k') \in T^{\opb}$.
The group $\autotopismgroup{T^{\opa}}$ is regular so
we may write $\nu_1$, $\nu_2$, $\nu_3 \in \autotopismgroup{T^{\opa}}$ for the 
unique autotopisms such that:
\begin{align*}
(i,j,k)   \nu_1 &= (i,j',k') \\
(i,j',k') \nu_2 &= (i',j',k') \\
(i,j',k') \nu_3 &= (i,j,k).
\end{align*}
By definition, $(i,j,k) \nu_1 \nu_2 \nu_3 = (i,j,k)$.
Since
$\autotopismgroup{T^{\opa}}$ is regular, it follows that $\nu_1 \nu_2
\nu_3=1$, so 
(Q3) is satisfied. Since the bitrade is nontrivial,
$\nu_i \neq 1$ for $i=1$, $2$, $3$.
By regularity each $\nu_i$ has no
fixed point, so (Q2) is satisfied.

Let $\rho$ be a cycle of $\nu_1$, and $\mu$ a cycle of $\nu_2$.
If $\left| \mov(\rho) \cap \mov(\mu) \right| > 1$ then
there exists an $(r,c,e) \in T^{\opa}$ such that
$(r,c,e) \nu_1^m = (r,c,e) \nu_2^n$
where 
$o(\nu_1) \nmid m$
and
$o(\nu_2) \nmid n$.
Then, $(r,c,e) a^m b^{-n} = (r,c,e)$ 
so by regularity, $a^m b^{-n} = 1$.
Since $\nu_1$ fixes row $i$ and $\nu_2$ fixes column $j$, it follows that
$a^m$ fixes row $i$ and column $j$. So $o(a) \mid m$, a contradiction.
By similar reasoning on $\nu_3$ we see that (Q1) is satisfied.

With (Q1), (Q2), and (Q3) it follows that
$[\nu_1,\, \nu_2,\, \nu_3]$ is a representation of
a bitrade $(T^{\opa},\, T^{\opbvar})$. Since $T^{\opa}$ is thin, it
follows that $(T^{\opa},\, T^{\opbvar}) = (T^{\opa},\, T^{\opb})$.
So $[\nu_1,\, \nu_2,\, \nu_3] = [\tau_1,\, \tau_2,\, \tau_3]$,
where $\tau_i$ is the representation of 
$(T^{\opa},\, T^{\opb})$.

Now let $G = \langle \nu_1, \nu_2, \nu_3 \rangle$ and define
three cyclic subgroups
$A = \langle \nu_1 \rangle$,
$B = \langle \nu_2 \rangle$, and
$C = \langle \nu_3 \rangle$. Create the group-based bitrade
$(U^{\opa},\, U^{\opb})$. Construct a bijection
$\theta \colon T^{\opa} \rightarrow U^{\opa}$.
\[
(i \alpha, j \beta, k \gamma) \mapsto (gA, gB, gC)
\textnormal{ where } (\alpha,\beta,\gamma)=g \in
\autotopismgroup{T^{\opa}}.
\]
Suppose that $i \alpha = i \alpha'$
where $(\alpha,\beta,\gamma) = g$ and $(\alpha',\beta',\gamma') = g'$.
Then $g^{-1} g'$ is an autotopism that fixes row $i$. In particular,
$(i,j,k) g^{-1} g' = (i,j',k')$
for some $j$, $j'$, $k$, $k'$. However,
$(i,j,k) \nu_1^m = (i,j',k')$ for some $m$.
By regularity it follows that
$g^{-1} g' = \nu_1^m \in A$, so
$gA = g'A$. By similar reasoning on columns and symbols
we see that $\theta$ is an isotopism. This establishes the theorem.  
\end{proof}

\section{Some examples}

Theorem~\ref{theoremSharplyTransitiveTrade} does not characterise all
group-based bitrades. For example the following group-based bitrade is
not thin.
\begin{align*}
&\begin{array}{c||ccc}
\opa & 1 & 2 & 3 \\
\hline \hline 1 & 1 & 2 & 3 \\
2 & 2 & 3 & 1 \\
3 & 3 & 1 & 2
\end{array}
\qquad
\begin{array}{c||ccc}
\opb & 1 & 2 & 3 \\
\hline \hline 1 & 1 & 2 & 3 \\
2 & 2 & 3 & 1 \\
3 & 3 & 1 & 2
\end{array} 
\qquad
\begin{array}{c||ccc}
\opa & ~ & ~ & ~ \\
\hline \hline ~ & 1 & 2 & 3 \\
~ & 4 & 5 & 6 \\
~ & 7 & 8 & 9
\end{array} \\
\tau_1 &= (1,3,2)(4,6,5)(7,9,8) \\
\tau_2 &= (1,4,7)(2,5,8)(3,6,9) \\
G &= \langle \tau_1, \tau_2 \rangle \\
\left| G \right| &= 9 = \left| X \right| \\
A &= C_{S_9}(G) = G
\end{align*}

\begin{example}
Let $A_1=\{a,b,c\}$, $A_2=\{d,e,f\}$ and $A_3=\{g,h,i,j\}$. 
Then $(T^{\opa},T^{\opb})$ is a {\em latin bitrade}, where
$T^{\opa},T^{\opb}\subset A_1\times A_2\times A_3$ are shown below:
\[
    T^{\opa} = \begin{array}{c|ccc}
    \opa & d & e & f \\
    \hline a & g & h & i \\
    b & h & i & j \\
    c & j & - & g \\
    \end{array}
    \qquad
    T^{\opb} = \begin{array}{c|ccc}
    \opb & d & e & f \\
    \hline a & h & i & g \\
    b & j & h & i \\
    c & g & - & j \\
    \end{array}
\]
We may also write:
\begin{eqnarray*}
T^{\opa} & = & \{(a,d,g),(a,e,h),(a,f,i),(b,d,h),(b,e,i),(b,f,j),
(c,d,j),(c,f,g)\}\hbox{ and } \\
T^{\opb} & = & \{(a,d,h),(a,e,i),(a,f,g),(b,d,j),(b,e,h),(b,f,i),
(c,d,g),(c,f,j)\}.
\end{eqnarray*} 
Let $\alpha_1 = (a b)$, $\alpha_2 = (d f)$, and $\alpha_3 = (g j)(h
i)$.  Then $\alpha = (\alpha_1,\, \alpha_2,\, \alpha_3)$ is an autotopism of
both $T^{\opa}$ and $T^{\opb}$.
Further,
$\autotopismgroup{T^{\opa}}=\autotopismgroup{T^{\opb}} = \langle \alpha \rangle$. Neither autotopism group is transitive since row $c$ is fixed.
\label{egg1}
\end{example} 

In general, however, it will not be true that
$\autotopismgroup{T^{\opa}}$ and
$\autotopismgroup{T^{\opb}}$ are equal. The next example shows a bitrade where
$\autotopismgroup{T^{\opa}}$ and $\autotopismgroup{T^{\opb}}$ are of a different order.

\begin{example}
Consider the following bitrade:
\[
T^{\opa} = 
\begin{array}{c|ccccc}
\opa &  0 & 1 & 2 & 3 & 4 \\
\hline
0&  0 & 1 & 2 & 3 & - \\
1&  1 & 4 & 3 & - & 2 \\
2&  2 & 3 & 1 & - & - \\
3&  3 & - & - & 0 & - \\
4&  - & 2 & - & - & 4
\end{array}
\qquad 
T^{\opb} = 
\begin{array}{c|ccccc}
\opb &  0 & 1 & 2 & 3 & 4 \\
\hline 
0&  2 & 3 & 1 & 0 & - \\
1&  3 & 1 & 2 & - & 4 \\
2&  1 & 2 & 3 & - & - \\
3&  0 & - & - & 3 & - \\
4&  - & 4 & - & - & 2
\end{array}
\]
We will show that $\autotopismgroup{T^{\opa}}$ is a nontrivial group while
$\autotopismgroup{T^{\opb}}$ contains only the identity. 
Let 
$\alpha_1 = (0,\, 1)(3,\, 4)$,
$\alpha_2 = (0,\, 1)(3,\, 4)$,
and
$\alpha_3 = (0,\, 4)(2,\, 3)$. Then
$\alpha = (\alpha_1,\, \alpha_2,\, \alpha_3)$ is an autotopism of
$T^{\opa}$, showing that $\autotopismgroup{T^{\opa}}$ is nontrivial.

On the other hand, suppose that
$(\alpha_1,\, \alpha_2,\, \alpha_3)$ is an autotopism of $T^{\opb}$.
Row~$2$ of $T^{\opb}$ is the only row of size~$3$ so $\alpha_1$ must fix
that row. Similarly, column~$2$ is the only column of size $3$ so
$\alpha_2$ must fix that column. This implies that $\alpha_3$ fixes the
symbols $1$, $2$, and $3$.

Only rows $0$ and $1$ are of size~$4$ so $\alpha_1$ may swap $0$ and $1$.
Similarly, $\alpha_2$ may swap columns $0$ and $1$.
So if $0 \alpha_1 = 1$ and $1 \alpha_1 = 0$ then there the top-left $2 \times 2$ subsquare of
$T^{\opb}$ will be transformed to
\[
\begin{array}{c|cc}
 &  0 & 1 \\
\hline 
0&  3 & 1 \\
1&  2 & 3
\end{array}
\qquad
\text{or}
\qquad
\begin{array}{c|cc}
 &  0 & 1 \\
\hline 
0&  1 & 3 \\
1&  3 & 2
\end{array}
\]
depending on whether $\alpha_2$ fixes or swaps the first two columns.
To complete the autotopism we would need
$3 \alpha_3 = 1$ for the left subsquare 
or
$2 \alpha_3 = 1$ for the right subsquare, and both are contradictions
since $\alpha_3$ fixes these points. So
we conclude that $\alpha_1$ fixes the first two rows of $T^{\opa}$.
Finally, consider rows $3$ and $4$. They are the only rows of size $2$
so $\alpha_1$ either fixes or swaps them. If $\alpha_1$ swaps these two rows
then, to get to the same shape as $T^{\opa}$, it must be that $\alpha_2$
swaps columns $0$ and $1$. Now the top-left $2 \times 2$ subsquare will
be
\[
\begin{array}{c|cc}
 &  0 & 1 \\
\hline 
0&  3 & 2 \\
1&  1 & 3
\end{array}
\]
which implies that $3 \alpha_3 = 1$, a contradiction.
Hence $\alpha_1$ is the identity.
To find a nontrivial autotopism we are now forced to use just $\alpha_2$
and $\alpha_3$. But any nontrivial $\alpha_2$ would require a nontrivial
$\alpha_1$ to get a partial latin square of the correct shape. So
$\alpha_2$
is also the identity.
Finally, the only solution to $(1,\, 1,\, \alpha_3) T^{\opa} = T^{\opa}$ 
is $\alpha_3 = 1$. We conclude that
$\autotopismgroup{T^{\opa}}$ is the trivial group.
\end{example}

\section{Bitrades of genus $0$}

\begin{theorem}
Let $(T^{\opa},T^{\opb})$ 
and
$(T^{\opa},T^{\opbvar})$
be separated bitrades of genus~$0$.
Then 
$T^{\opb} = T^{\opbvar}$.
\end{theorem}

\begin{proof}
Let $(T^{\opa},\, T^{\opb})$ and $(T^{\opa},\, T^{\opbvar})$ be
separated bitrades of genus~$0$ where
$T^{\opb} \neq T^{\opbvar}$. Let $\mathcal{G}^{\opb}$
and $\mathcal{G}^{\opbvar}$ be the hypermap embeddings of $(T^{\opa},\,
T^{\opb})$ and $(T^{\opa},\, T^{\opbvar})$.

Pick some row $r_0$ where the order of shaded triangles in 
$\mathcal{G}^{\opbvar}$ is different to that of $\mathcal{G}^{\opb}$.
Let $r_1, r_2, \dots, r_t$ be the nearest row vertices to $r_0$. Now
choose some $(r_i, c_j, s_k) \in T^{\opa}$ such that
$r_i \notin \{ r_0, r_1, r_2, \dots, r_t \}$. (If it is not possible to
choose such an entry then apply a conjugate argument to columns symbols
or symbols. If even that is impossible then the bitrade must be the
intercalate and the theorem follows.) By removing the face
$(r_i, c_j, s_k)$ in $\mathcal{G}^{\opbvar}$ we can draw the hypermap on
the plane (refer to this embedding in the plane by
$\overline{\mathcal{G}}^{\opbvar}$).

In $\overline{\mathcal{G}}^{\opbvar}$ there will be a cycle
\begin{equation}\label{eqnBadCycle}
c_1, s_1, c_2, s_2, \dots, c_m, s_m, c_{m+x}, s_{m+x}, \dots, c_{m+1}, s_{m+1}, \dots, c_1, s_1
\end{equation}
where $x > 1$. Since this cycle is not the normal increasing sequence as
in $\overline{\mathcal{G}}^{\opb}$, there exists an integer $n$ such that
$s_n$ is between $s_{m+x}$ and $c_{m+1}$,
and
$c_{n+1}$ is between $s_{m+1}$ and $c_{1}$, moving left to right in 
\eqref{eqnBadCycle}.

Due to $T^{\opa}$ there are edges 
$(s_m, c_{m+1})$ and $(s_n, c_{n+1})$. However it is not possible to
draw both of these edges in $\overline{\mathcal{G}}^{\opbvar}$ while 
retaining planarity, a contradiction.
\end{proof}

A very similar proof gives the following result:

\begin{corollary}\label{thmEqualAutotopismGroups}
If $(T^{\opa},T^{\opb})$ is a separated bitrade of genus $0$, 
then $\autotopismgroup{T^{\opa}} = \autotopismgroup{T^{\opb}}$. 
\end{corollary}

\section{Automorphisms and autotopisms}

\begin{lemma}\label{lemmaAutotopismInAutomorphism}
Let 
$(T^{\opa},T^{\opb})$ be a separated bitrade with representation 
$[\tau_1,\, \tau_2,\, \tau_3]$.
Then $\automorphismgroup{\tau_i}$ is isomorphic to a subgroup of
$\autotopismgroup{T^{\opa},T^{\opb}}$.
\end{lemma}

\begin{proof}
Let $n = \left| T^{\opa} \right|$.
A necessary and sufficient condition for 
$\theta \in S_n$ be in $\automorphismgroup{\tau_i}$ is that
$\theta$ commutes with $\tau_i$, for $i = 1$, $2$, $3$. 
Consider a cycle $\rho=(x,y,\ldots)$ in $\tau_1$. Then
the condition
$x \theta \tau_1 = x \tau_1 \theta$ implies that
$(x \theta) \tau_1 = (y \theta)$
so $\theta$ sends cycles of $\tau_1$ to cycles of $\tau_1$.
Thus $\theta$ acts as a permutation on the set ${\mathcal A}_i$ of
cycles of $\tau_1$. In this way, $\theta$ is an autotopism (say,
$\overline{\theta}$) of
$(T^{\opa},\, T^{\opb})$. We may think of
${\mathcal A}_1 \cup {\mathcal A}_2 \cup {\mathcal A}_3$
as a system of blocks for $\automorphismgroup{\tau_i}$ and
the mapping $\theta \mapsto \overline{\theta}$ is a homomorphism into
$\autotopismgroup{T^{\opa},\, T^{\opb}}$. 
So $\automorphismgroup{\tau_i}$ is isomorphic to a subgroup of
$\autotopismgroup{T^{\opa},T^{\opb}}$.
\end{proof}

Question: when is 
$\automorphismgroup{\tau_i}$ isomorphic to $\autotopismgroup{T^{\opa},T^{\opb}}$?

\begin{remark} 
In general $\automorphismgroup{\tau_i}$ may be isomorphic to a proper subgroup of
$\autotopismgroup{T^{\opa}}$:
\begin{align*}
\begin{aligned}
    T^{\opa} &= 
    \begin{array}{c||ccccc}
    \opa & 0 & 1 & 2 & 3 & 4 \\
    \hline \hline
    0 & 0 & 1 & 2 & 3 & 4 \\
    1 & 1 & 2 & 3 & 4 & 0 \\
    2 & 2 & 3 & 4 & 0 & 1 \\
    3 & 3 & 4 & 0 & 1 & 2 \\
    4 & 4 & 0 & 1 & 2 & 3
    \end{array}
\end{aligned}
\begin{aligned}
    T^{\opb} &=
    \begin{array}{c||ccccc}
    \opb & 0 & 1 & 2 & 3 & 4 \\
    \hline \hline
    0 & 4 & 0 & 1 & 2 & 3  \\
    1 & 0 & 3 & 2 & 1 & 4 \\
    2 & 1 & 4 & 0 & 3 & 2 \\
    3 & 2 & 1 & 3 & 4 & 0 \\
    4 & 3 & 2 & 4 & 0 & 1
    \end{array}
\end{aligned}
\end{align*}
Then $T^{\opa}$ has an autotopism $\theta$ defined by
$(x,y,z) \theta = (x+1, y, z+1)$ where addition is modulo $5$.
By labelling the elements in the top two rows of $T^{\opa}$ with the letters $a$, $b$,
\ldots, $j$, we see that $\tau_1$ consists of a single $5$-cycle for the
first row, and two disjoint cycles for the second row:
\begin{align*}
\begin{aligned}
\begin{array}{ccccc}
a & b & c & d & e \\
f & g & h & i & j
\end{array}
\end{aligned}
\qquad \tau_1 = (a, b, c, d, e) (f, i, j)(g, h) \cdots
\end{align*}
Now $\theta$ can be interpreted as a map on the letters, for example
\[
(1,1,1) \theta = (2,1,2) \quad \Rightarrow \quad a \theta = f.
\]
If $\theta$ is an element of $\automorphismgroup{\tau_i}$ then it must (at least) commute with
$\tau_1$:
\begin{align*}
a \tau_1 \theta &= a \theta \tau_1 
\Rightarrow b \theta = f \tau_1 
\Rightarrow g = i, 
\end{align*}
which is a contradiction. So $\theta$ is not an automorphism.  
\end{remark}

\begin{lemma}[Lemma 1.7.14, \cite{LZ}]
\label{lemmaLeftRightAction}
Let a group $G$ of order $n$ act on itself by multiplication on the
right: an element $x \in G$ acts by sending $a \in G$ to $ax$. Then the
centraliser of $G$ in $S_n$ is obtained by the action of $G$ on itself
by multiplication on the left: an element $y \in G$ acts by sending $a
\in G$ to $y^{-1}a$.
\end{lemma}

Let $g_i$ be permutations acting on a set $\darts$ of size $n$.
A {\em $k$-constellation} is a sequence $C = [g_1, g_2, \ldots, g_k]$ where
$g_i \in S_n$ which satisfies two properties:
\begin{itemize}
\item the cartographic group $G = \langle g_1, g_2, \ldots, g_k \rangle$ acts
transitively on the underlying set of $n$ points;
\item $g_1 g_2 \cdots g_k = \text{id}$.
\end{itemize}
We say that $[g_1, g_2, \ldots, g_k]$ has degree $n$ and length $k$.
The {\em automorphism group} of a constellation $C = [g_1, g_2, \ldots,
g_k]$ is the group
\[
\text{Aut}(C) = \{ h \in S_n \mid h^{-1} g_i h = g_i \text{, for all $i = 1, \ldots, k$}\}
\]
which is just $C_{S_n}(G)$, the centraliser of $G$ in $S_n$.
In this terminology, a latin bitrade is a $3$-constellation $[\tau_1, \tau_2, \tau_3]$
that also satisfies \textnormal{(Q1)} and \textnormal{(Q2)}.

\begin{proposition}[Proposition 1.7.15, \cite{LZ}]
\label{propTransitiveIsomorphic}
The following two properties of a constellation are equivalent:
\begin{enumerate}
\item The automorphism group acts transitively on the underlying set.
\item The automorphism group is isomorphic to the cartographic group.
\end{enumerate}
If a constellation possesses either (and then both) of these properties,
then the elements of the underlying set are in bijection with the
elements of the cartographic group, and the group acts on itself by
multiplication on the right, while the automorphism group acts by
multiplication on the left, as in Lemma~\ref{lemmaLeftRightAction}.
\end{proposition}

\begin{lemma}\label{lemmaSharplyIsomorphic}
If $\autotopismgroup{T^{\opa}}$ is regular and
$\automorphismgroup{\tau_i}$ is transitive then 
$\autotopismgroup{T^{\opa}}$ is isomorphic to
$\automorphismgroup{\tau_i}$.
\end{lemma}

\begin{proof}
The autotopism group $\autotopismgroup{T^{\opa}}$ is regular so
$\left| \autotopismgroup{T^{\opa}} \right| = \left| T^{\opa} \right|$.
By Lemma~\ref{lemmaAutotopismInAutomorphism},
$\automorphismgroup{\tau_i}$ is isomorphic to a subgroup of $\autotopismgroup{T^{\opa}}$ so
$\left| \automorphismgroup{\tau_i}\right| \leq \left| \autotopismgroup{T^{\opa}} \right| = \left|
T^{\opa} \right|$. Since
$\autotopismgroup{T^{\opa}}$ is transitive, Proposition~\ref{propTransitiveIsomorphic}
implies that 
$\left| \automorphismgroup{\tau_i} \right| = \left| T^{\opa} \right|$. Hence 
$\autotopismgroup{T^{\opa}} \cong \automorphismgroup{\tau_i}$.  
\end{proof}


\begin{thebibliography}{1}

\bibitem{MR2041871}
Elizabeth~J. Billington.
\newblock Combinatorial trades: a survey of recent results.
\newblock In {\em Designs, 2002}, volume 563 of {\em Math. Appl.}, pages
  47--67. Kluwer Acad. Publ., Boston, MA, 2003.

\bibitem{cavenaghSurvey}
Nicholas~J.\ Cavenagh.
\newblock The theory and application of latin bitrades: a survey.
\newblock submitted, 2008.

\bibitem{groupsarxiv}
Nicholas~J. Cavenagh, Ale\v s~Dr\'apal, and Carlo H\"{a}m\"{a}l\"{a}inen.
\newblock Latin bitrades derived from groups, 2007.
\newblock {\tt arXiv:0710.0938v2 [math.CO]}, to appear in Discrete Mathematics.

\bibitem{Dr9}
A~Dr{\'a}pal.
\newblock Geometry of latin trades.
\newblock Manuscript circulated at the conference Loops, Prague, 2003.

\bibitem{hamalainen2007}
Carlo H\"{a}m\"{a}l\"{a}inen.
\newblock {\em Latin Bitrades and Related Structures}.
\newblock {PhD} in {M}athematics, Department of Mathematics, The University of
  Queensland, 2007.
\newblock http://carlo-hamalainen.net/phd/hamalainen-20071025.pdf.

\bibitem{3hom}
Carlo H\"{a}m\"{a}l\"{a}inen.
\newblock Partitioning $3$-homogeneous latin bitrades.
\newblock {\em Geometriae Dedicata}, 133(1):181--193, 2008.

\bibitem{LZ}
Sergei~K. Lando and Alexander~K. Zvonkin.
\newblock {\em Graphs on surfaces and their applications}, volume 141 of {\em
  Encyclopaedia of Mathematical Sciences}.
\newblock Springer-Verlag, Berlin, 2004.
\newblock With an appendix by Don B. Zagier, Low-Dimensional Topology, II.

\end{thebibliography}

\end{document}